\documentclass[11pt]{article}
\usepackage{amsmath, amsfonts, amssymb}
\usepackage{graphicx, amsthm, color}
\usepackage{cite}

\newtheorem{theorem}{Theorem}[section]
\newtheorem{lemma}[theorem]{Lemma}

\theoremstyle{definition}

\newtheorem{example}[theorem]{Example}

\theoremstyle{remark}
\newtheorem{remark}[theorem]{Remark}
\numberwithin{equation}{section}
\newcommand{\R}{\mathbb R}

\def\e{\varepsilon}

\definecolor{vert}{rgb}{0,0.4,0}

\begin{document}

\begin{center}
{\Large Explicit formulas for $C^{1,1}$ Glaeser-Whitney extensions \medskip
of $1$-Taylor fields in Hilbert spaces.}
\end{center}

\smallskip

\begin{center}
\textsc{Aris Daniilidis, Mounir Haddou, Erwan Le Gruyer, Olivier Ley}
\end{center}

\medskip

\noindent\textbf{Abstract.} We give a simple alternative proof for the
$C^{1,1}$--convex extension problem which has been introduced and studied by
D.~Azagra and C.~Mudarra \cite{am17}. As an application, we obtain an easy
constructive proof for the Glaeser-Whitney problem of $C^{1,1}$ extensions on
a Hilbert space. In both cases we provide explicit formulae for the extensions.
For the Glaeser-Whitney problem the obtained extension is almost minimal,
that is, minimal up to a multiplicative factor in the sense of Le Gruyer~\cite{legruyer09}.

\bigskip

\noindent\textbf{Key words.} Whitney extension problem, convex extension,
sup-inf convolution, semiconvex function.

\vspace{0.6cm}

\noindent\textbf{AMS Subject Classification} \ \textit{Primary} 54C20 ;
\textit{Secondary} 52A41, 26B05, 26B25, 58C25.

\section{Introduction}

Determining a function (or a class of functions) of a certain regularity
fitting to a prescribed set of data is one of the most challenging problems in
modern mathematics. The origin of this problem is very old, since this
general framework encompasses classical problems of applied analysis.
Depending on the requested regularity, it
goes from the Tietze extension theorem in normal topological spaces, where the
required regularity is minimal (continuity), to results where the requested
regularity is progressively increasing: McShane results on uniformly
continuous, H\"{o}lder or Lipschitz extensions~\cite{mcshane34}, Lipschitz
extensions for vector-valued functions (Valentine~\cite{Valentine45}),
differentiable and $C^{k}$-extensions (Whitney~\cite{whitney34},
Glaeser~\cite{glaeser58}, and more recently Brudnyi-Shvartsman~\cite{bs01},
Zobin~\cite{zobin98}, Fefferman~\cite{fefferman05}),
monotone multivalued extensions (Bauschke-Wang~\cite{bw10}), definable
(in some o-minimal structure) Lipschitz extensions (Aschenbrenner-Fischer~\cite{af11}), etc.
In this work
we are interested in the Glaeser-Whitney $C^{1,1}$-extension problem,
which we describe below.\smallskip

Let $S$ be a nonempty subset of a Hilbert space $(\mathcal{H},\langle \cdot,\cdot\rangle, |\cdot|)$ and assume
$\alpha:S\rightarrow\mathbb{R}$ and $v:S\rightarrow\mathcal{H}$ satisfy the
so-called Glaeser-Whitney conditions:

\begin{eqnarray}
  \label{cond-GW}
\left\{
\begin{array}
[c]{l}
\displaystyle \mathop{\rm sup}_{s_1,s_2\in S, s_1\not= s_2}
\frac{|\alpha(s_{2})-\alpha(s_{1})-\langle v(s_{1}),s_{2}-s_{1}\rangle|}{|s_{1}-s_{2}|^{2}}
:= K_1 <+\infty,   \bigskip 
\\
\displaystyle \mathop{\rm sup}_{s_1,s_2\in S, s_1\not= s_2}
\frac{|v(s_{1})-v(s_{2})|}{|s_{1}-s_{2}|}:= K_2<+\infty. 
\end{array}
\right.
\end{eqnarray}

In \cite{whitney34, glaeser58} it has been shown that under the above
conditions, in case $\mathcal{H}=\mathbb{R}^{n},$ there exists a $C^{1,1}$-smooth
function $F:\mathbb{R}^{n}\rightarrow\mathbb{R}$ such that the
prescribed $1$-Taylor field $(\alpha(s),v(s))$ coincides, at every
$s\in S$, with the $1$-Taylor field $(F(s),\nabla F(s))$ of $F$. The above
result has been extended to Hilbert spaces in Wells~\cite{wells73}
and Le~Gruyer~\cite{legruyer09}. In particular, in~\cite{legruyer09} the following constant
has been introduced:
\begin{eqnarray}\label{gamma1}
&& \Gamma^{1}(S,(\alpha,v)):= \mathop{\rm sup}_{s_1,s_2\in S, s_1\not= s_2} \left(\sqrt{A_{s_1 s_2}^2+B_{s_1 s_2}^2}+|A_{s_1 s_2}|\right),
\end{eqnarray}
where
\begin{eqnarray*}
A_{s_1 s_2}=\frac{2(\alpha(s_1)-\alpha(s_2))+\langle v(s_1)+v(s_2), s_2-s_1\rangle} {|s_1-s_2|^2},
\quad
B_{s_1 s_2}=\frac{v(s_1)-v(s_2)}{|s_1-s_2|}.
\end{eqnarray*}
It has been shown in~\cite{legruyer09} that $\Gamma^{1}(S,(\alpha,v))<+\infty$ if and only if conditions~\eqref{cond-GW}
hold. Moreover, in this case, the existence of a $C^{1,1}$ function $F:\mathcal{H}\rightarrow\mathbb{R}$
such that $F|_{S}=\alpha,$ $\nabla F|_{S}=v$ and
\begin{equation}
\Gamma^{1}(\mathcal{H},(F,\nabla F))=\Gamma^{1}(S,(\alpha,v)),\label{1}
\end{equation}
has been established. Henceforth, every $C^{1,1}$-extension of $(\alpha,v)$ satisfying \eqref{1} will be called a minimal
Glaeser-Whitney extension. The terminology is justified by the fact
that, for every $C^{1,1}$ function $G:\mathcal{H}\rightarrow\mathbb{R},$
we have $\Gamma^{1}(\mathcal{H},(G,\nabla G))=\mathrm{Lip}(\nabla G)$ (see \cite[Proposition 2.4]{legruyer09}). Thus $\mathrm{Lip}(\nabla F)\leq \mathrm{Lip}(\nabla G)$
for any $C^{1,1}$-extension $G$ of the prescribed $1$-Taylor field $(\alpha(s),v(s))$. If for some universal constant $K \geq 1$ (not depending on the data) we have $\Gamma^{1}(\mathcal{H},(G,\nabla G)) \leq K \, \Gamma^{1}(S,(\alpha,v))$ then the extension $G$ will be called almost minimal. \smallskip

Recently, several authors have been interested in extensions that are subject to additional constraints:
extensions which preserve positivity~\cite{fil17, fil17bis} or convexity~\cite{am17, am17bis}.
In \cite{am17}, D. Azagra and C. Mudarra considered the problem
of finding a {\em convex} $C^{1,1}$-smooth extention over a prescribed Taylor
polynomial $(\alpha(s),v(s))_{s\in S}$ in a Hilbert space
$\mathcal{H}$ and established that the condition
\begin{equation}
\alpha(s_{2})\geq\alpha(s_{1})+\langle v(s_{1}),s_{2}-s_{1}\rangle+\frac
{1}{2M}|v(s_{1})-v(s_{2})|^{2},\quad\forall s_{1},s_{2}\in
S,\label{form-azagra}
\end{equation}
is necessary and sufficient for the existence of such extension.\smallskip

Inspired by the recent work~\cite{am17}
concerning $C^{1,1}$-convex extensions, we revisit the classical
Glaeser-Whitney problem. We first
provide an alternative shorter proof of the result of~\cite{am17}
concerning $C^{1,1}$-convex extensions in Hilbert spaces
by giving a simple explicit formula. This formula is heavily based
on the regularization via sup-inf convolution in the spirit of Lasry-Lions~\cite{ll86}
and can be efficiently  computed, see Remark~\ref{rmk-num}.
As an easy consequence, we obtain 
a direct proof for the classical
$C^{1,1}$-Glaeser-Whitney problem in Hilbert spaces, which goes together
with an explicit formula of the same type as for the convex extension problem.
Let us mention that the previous proofs are quite involved
  both  in finite dimension~\cite{whitney34,glaeser58}
  and Hilbert spaces~\cite{wells73, legruyer09}.
In the finite dimensional case, a construction of the extension is proposed
  in~\cite{wells73} and some  explicit formulae can be found in~\cite{lgp15}
  but both are not tractable (see however the work~\cite{hhm17} for
  concrete computations).
Our approach also compares favorably to the result of~\cite{legruyer09}, in which the existence
of minimal extensions is established.
On the other hand, the extension given by our explicit formula may fail to be minimal
---though it is almost minimal up to a universal multiplicative factor.
\smallskip

Before we proceed, we recall that a function $f:\mathcal{H}
\rightarrow\mathbb{R}$ is called $C_{\ast}$-semiconvex (resp., $C^{\ast}$-semiconcave) when, for all $x,y\in H,$
\[
f(y)-f(x)-\langle\nabla f(x),y-x\rangle\geq-\frac{C_{\ast}}{2}|x-y|^{2}
\ \ \text{(resp., $\displaystyle\leq\frac{C^{\ast}}{2}|x-y|^{2}$).}
\]
This is equivalent to assert that $f+\frac{C_{\ast}}{2}|x|^{2}$ is convex
(respectively $f-\frac{C^{\ast}}{2}|x|^{2}$ is concave). When $f$ is both $C
$-semiconvex and $C$-semiconcave, then $f$ is $C^{1,1}$ in $\mathcal{H}$ with
$\mathrm{Lip}(\nabla f)\leq C$ (for a proof of this latter result in finite
dimension, see~\cite{cs04} and use the arguments of~\cite{ll86} to extend the
result to Hilbert spaces).~~~
\smallskip


\section{Convex $C^{1,1}$ extension of $1$-fields}
\label{ss-2}

For any $f:\mathcal{H}\rightarrow\mathbb{R}$ and $\varepsilon>0,$ we define
respectively the sup and the inf-convolution of $f$ by
\[
f^{\varepsilon}(x)=\mathop{\rm sup}_{y\in\mathcal{H}}\left \{ f(y)-\frac{|y-x|^{2}}{2\varepsilon} \right \},\quad f_{\varepsilon}(x)=\mathop{\rm inf}_{y\in\mathcal{H} }\left \{f(y)+\frac{|y-x|^{2}}{2\varepsilon} \right \}.
\]

\begin{theorem}
[C$^{1,1}$-convex extension]\label{thm-main-1} Let $S$ be any nonempty subset
of the Hilbert space $\mathcal{H}$ and $(\alpha(s),v(s))_{s\in S}$ be a
1-Taylor field on $S$ satisfying~\eqref{form-azagra} for some constant $M>0.$
Then
\begin{equation}
f(x)=\mathop{\rm sup}_{s\in S}\;\{\alpha(s)+\langle v(s),x-s\rangle
\}\label{fcvx}
\end{equation}
is the smallest continuous convex extension of $(\alpha,v)$ in $\mathcal{H}$
and
\begin{equation}
F(x)=
\mathop{\rm lim}_{\varepsilon\nearrow\frac{1}{M}}
(f^{\varepsilon})_{\varepsilon}(x)
=
\mathop{\rm lim}_{\varepsilon\nearrow\frac{1}{M}}
\mathop{\rm inf}_{z\in\mathcal{H}}\mathop{\rm sup}_{y\in\mathcal{H}
}\{f(y)-\frac{|y-z|^{2}}{2\varepsilon}+\frac{|z-x|^{2}}{2\varepsilon
}\}\label{ext-C11-f}
\end{equation}
is a $C^{1,1}$ convex extension of $(\alpha,v)$ in
$\mathcal{H}.$ Moreover, $\mathrm{Lip}(\nabla F)\leq M.$
\end{theorem}

\begin{remark}\label{rmk-num}\  (i) The function $f$ given by~\eqref{fcvx} is the smallest convex continuous extension of  $(\alpha,v)$
in the following sense: if $g$ is a continuous convex function in $\mathcal{H}$, differentiable on $S$, satisfying $g(s)=\alpha(s)$ and $\nabla g(s)=v(s)$, for all $s \in S$, then $f\leq g.$ \smallskip 

\noindent (ii) As we shall see in the forthcoming proof, $\varepsilon\mapsto(f^{\varepsilon})_{\varepsilon}$ is
nondecreasing. Therefore, ``$\mathop{\rm lim}_{\varepsilon\nearrow\frac{1}{M}}$'' can be replaced by
``$\mathop{\rm sup}_{\varepsilon\in(0,\frac{1}{M})}$'' in formula~\eqref{ext-C11-f}. \smallskip

\noindent (iii) The inf-convolution corresponds to the well-known Moreau-Yosida  re\-gu\-la\-rization in convex analysis.
It is also related to the Legendre-Fenchel transform (convex conjugate). A discussion on theoretical and practical properties of this regularization can be found in~\cite{clcs97}  and references therein. In practice, $f_{\varepsilon}$, $f^{\varepsilon}$ and therefore the formula~\eqref{ext-C11-f} can be very
efficiently computed using different techniques and algorithms such as~\cite{BH09} or~\cite{yL06}.
\end{remark}

\begin{proof}[Proof of Theorem~\ref{thm-main-1}]
For all $x\in\mathcal{H}$ and $s_{1},s_{2}\in S,$
by~\eqref{form-azagra}, we have
\begin{align*}
&  \alpha(s_{1})+\langle v(s_{1}),x-s_{1}\rangle\\
&  \leq\alpha(s_{2})+\langle v(s_{2}),x-s_{2}\rangle+\langle v(s_{1}
)-v(s_{2}),x-s_{2}\rangle-\frac{1}{2M}|v(s_{1})-v(s_{2})|^{2}\\
&  \leq\alpha(s_{2})+\langle v(s_{2}),x-s_{2}\rangle+\mathop{\rm sup}_{\xi
\in\mathcal{H}}\{\langle\xi,x-s_{2}\rangle-\frac{1}{2M}|\xi|^{2}\}\\
&  =\alpha(s_{2})+\langle v(s_{2}),x-s_{2}\rangle+\frac{M}{2}|x-s_{2}|^{2}.
\end{align*}
It follows that for all $x\in\mathcal{H}$ and $s\in S$
\begin{equation}\label{f-line-quad}
\alpha(s)+\langle v(s),x-s\rangle\leq f(x)\leq\alpha(s)+\langle
v(s),x-s\rangle+\frac{M}{2}|x-s|^{2}.
\end{equation}
In particular, the function $f$ defined by \eqref{fcvx} is convex, finite in
$\mathcal{H}$ and trapped between affine hyperplanes and quadratics with
equality on $S$. Therefore, it is differentiable on $S$ with $f(s)=\alpha(s)$,
$\nabla f(s)=v(s)$ and it is clearly the smallest continuous
convex extension of the field.\smallskip

Setting $q(x)=\alpha(s)+\langle v(s), x-s\rangle+\frac{M}{2}|x-s|^{2},$ for
$\varepsilon\in(0,M^{-1}),$ straightforward computations lead to the formulae:
\begin{align}
&  q^{\varepsilon}(x)=\alpha(s)+\frac{1}{1-\varepsilon M}\left(  \frac{M}{2}
  |x-s|^{2}+\langle v(s),x-s\rangle+\frac{\varepsilon}{2}|v(s)|^{2}\right),
\label{sup-conv-quad}\\
&  q_{\varepsilon}(x)=  \alpha(s)+\frac{1}{1+\varepsilon M}\left(  \frac{M}{2}|x-s|^{2}
+\langle v(s),x-s\rangle-\frac{\varepsilon}{2}|v(s)|^{2}\right).\nonumber
\end{align}
In particular, after a new short computation, we deduce that
\begin{equation}\label{form-cruciale}
(q^{\varepsilon})_{\varepsilon}=q,
\end{equation}
and from~\eqref{f-line-quad}, since the sup and inf-convolution are
order-preserving operators, we obtain that for every $\varepsilon\in
(0,M^{-1}),$ $x\in\mathcal{H}$ and $s\in S,$
\begin{equation}
\alpha(s)+\langle v(s),x-s\rangle\leq(f^{\varepsilon})_{\varepsilon}
(x)\leq\alpha(s)+\langle v(s),x-s\rangle+\frac{M}{2}|x-s|^{2}.\label{fepseps}
\end{equation}
It follows that $(f^{\varepsilon})_{\varepsilon}$ is well-defined on
$\mathcal{H}$. Notice also that
\begin{equation}
f\leq (f^{\varepsilon})_{\varepsilon}\qquad\text{ in $\mathcal{H}$
}\label{ordre-sup-inf}
\end{equation}
and that $(f^{\varepsilon})_{\varepsilon}$ is differentiable on $S$ with
$(f^{\varepsilon})_{\varepsilon}(s)=\alpha(s)$ and $\nabla(f^{\varepsilon})_{\varepsilon}(s)=v(s)$,
for every $s\in S$.\smallskip

Notice that since $f$ is defined as the supremum of the affine functions
$\ell_s(x)=\alpha(s)+\langle v(s),x-s\rangle$ and
$\ell_s^{\varepsilon}(x)=\ell_s(x)+\frac{\varepsilon}{2}|v(s)|^{2}$ by~\eqref{sup-conv-quad},
we have
\[
f^{\varepsilon}(x)=\mathop{\rm sup}_{s\in S}\{\ell_s^{\varepsilon}(x)\},
\]
which proves that $f^{\varepsilon}$ is convex. Therefore, $(f^{\varepsilon
})_{\varepsilon}$ is still convex, being the infimum with respect to $y$ of the
jointly convex functions
\[
f^{\varepsilon}(y)+\frac{1}{2\varepsilon}|y-x|^{2},\quad(x,y)\in
\mathcal{H}\times\mathcal{H}.
\]
It is well-known~\cite{ll86} that the sup and inf-convolution satisfy some
semigroup properties,
\[
f^{\varepsilon+\varepsilon^{\prime}}=(f^{\varepsilon})^{\varepsilon^{\prime}
}\text{ and }f_{\varepsilon+\varepsilon^{\prime}}=(f_{\varepsilon
})_{\varepsilon^{\prime}}\text{ for all $\varepsilon,\varepsilon^{\prime}>0.$}
\]
Therefore, for $0<\varepsilon<\varepsilon^{\prime},$ $f^{\varepsilon^{\prime}
}=(f^{\varepsilon})^{\varepsilon^{\prime}-\varepsilon}.$
By~\eqref{ordre-sup-inf}, we infer $((f^{\varepsilon})^{\varepsilon^{\prime
}-\varepsilon})_{\varepsilon^{\prime}-\varepsilon}\geq f^{\varepsilon}.$ It
follows
\[
((f^{\varepsilon^{\prime}})_{\varepsilon^{\prime}-\varepsilon})_{\varepsilon
}=(f^{\varepsilon^{\prime}})_{\varepsilon^{\prime}}\geq(f^{\varepsilon
})_{\varepsilon}\quad\text{for all $0<\varepsilon<\varepsilon^{\prime}.$}
\]

We conclude that $\varepsilon\mapsto(f^{\varepsilon})_{\varepsilon}$ is
nondecreasing on $(0,M^{-1})$ so $F$ is well defined, convex and still
satisfies~\eqref{fepseps}. Therefore $F$ is an extension of $(\alpha
(s),v(s))_{s\in S}$ in $\mathcal{H}$ and is differentiable on $S.$\smallskip

It remains to prove that $F$ is $C^{1,1}$ in $\mathcal{H}$ and to estimate
$\mathrm{Lip}(\nabla F).$ From~\cite{ll86}, we know that the inf-convolution
$(f^{\varepsilon})_{\varepsilon}$ of $f^{\varepsilon}$ is $\varepsilon^{-1}
$-semiconcave. Since $(f^{\varepsilon})_{\varepsilon}$ is also convex, it
means that $(f^{\varepsilon})_{\varepsilon}$ is both $\varepsilon^{-1}
$-semiconcave and $\varepsilon^{-1}$-semiconvex. Therefore $(f^{\varepsilon
})_{\varepsilon}$ is $C^{1,1}$ in $\mathcal{H}$ with $\mathrm{Lip}
(\nabla(f^{\varepsilon})_{\varepsilon})\leq\varepsilon^{-1}. $ Since
$(f^{\varepsilon})_{\varepsilon}-\frac{1}{2\varepsilon}|x|^{2}$ is concave for
every $0<\varepsilon<M^{-1},$ sending $\varepsilon\nearrow M^{-1},$ we
conclude that $F$ is $M$-semiconcave. Since $F$ is also convex, the previous
arguments allow to conclude that $F$ is $C^{1,1}$ in $\mathcal{H}$ with
$\mathrm{Lip}(\nabla F)\leq M.$
\end{proof}

\begin{remark} In~\cite{ll86}, the $C^{1,1}$ regularization result is stated for $(f^\e)_\delta$ with $0<\delta <\e.$ To obtain an extension
in our framework, we need to take $\delta =\e$. The fact that we have been able to increase the value of $\delta$ and take it equal to $\e$ 
without losing the $C^{1,1}$ regularity relies strongly on the convexity of $f$. Since convexity is preserved under the sup and inf-convolution operations,
the inf-convolution does not affect the semiconvexity property of $f^{\varepsilon}$ even for $\delta=\e.$
For the same reason, one cannot reverse the above operations: more precisely, the function
$(f_{\varepsilon})^{\varepsilon}=f$ would not be semiconcave.
\end{remark}


\section{$C^{1,1}$ extension of $1$-fields: explicit formulae}

Let us now apply the previous result to obtain a general $C^{1,1}$-extension
in the Glaeser-Whitney problem.

\begin{theorem} [C$^{1,1}$-Glaeser-Whitney almost minimal extension]\label{thm-main-2} Let $S$
be a nonempty subset of a Hilbert space $\mathcal{H}$ and $(\alpha(s),v(s))_{s\in S}$ be a $1$-Taylor field on $S$ satisfying~\eqref{cond-GW}.
Then, the function
$$
G(x)=F(x)-\frac{1}{2}\overline{\mu}|x|^2
$$
is an explicit $C^{1,1}$ extension of the $1$-Taylor field $(\alpha,v)$, provided that $F$ is 
the convex extension of the $1$-Taylor field $(\tilde{\alpha},\tilde{v})$ where, for all $s\in S$,
$$
\tilde{\alpha}(s):=\alpha(s)+\frac{1}{2}\overline{\mu}|s|^{2},
\quad
\tilde{v}(s):=v(s)+\overline{\mu}s
$$
and
$$
  \overline{\mu}:= 2K_1+K_2+\sqrt{(2K_1+K_2)^2+K_2^2},
  \quad\text{$K_1,K_2$ given by~\eqref{cond-GW}.}
$$

Moreover, the extension $G$ is almost minimal, i.e.,
$$
\Gamma^{1}(S,(\alpha,v))\leq \Gamma^{1}(\mathcal{H},(G,\nabla G))=\mathrm{Lip}(\nabla G)
\leq\left(\frac{5+\sqrt{29}}{2}\right)\Gamma^{1}(S,(\alpha,v)).
$$
\end{theorem}

\begin{proof}[Proof of Theorem~\ref{thm-main-2}]
We check that for every $\mu>2K_{1}$ the 1-Taylor field $$(\tilde{\alpha}(s), \tilde{v}(s)):=(\alpha(s)+\frac{\mu}{2}|s|^{2}, v(s)+\,\mu s)$$
satisfies~\eqref{form-azagra} with $M=(\mu +K_{2})^{2}(\mu-2K_1)^{-1}$. Indeed, for any $s_{1},s_{2}\in S$ we obtain, using~\eqref{cond-GW},
\begin{eqnarray*}
  && \tilde{\alpha}(s_{2})- \tilde{\alpha}(s_{1})- \,\langle\tilde{v}(s_{1}),s_{2}-s_{1}\rangle\\
&=& \alpha(s_{2})-\alpha(s_{1})-\langle v(s_{1}),s_{2}-s_{1}\rangle + \frac{\mu}{2}\,\left( |s_{2}|^{2}-|s_{1}|^{2}-2 \langle s_{1},s_{2}-s_{1}\rangle\right) \\
& \geq &  \left( \frac{\mu-2K_{1}}{2}\right) \;|s_{1}-s_{2}|^{2}
\geq \frac{1}{2}(\frac{\mu-2K_{1}}{(\mu +K_{2})^{2}})\;|\tilde{v}(s_{1})-\tilde{v}(s_{2})|^{2},
\end{eqnarray*}
since ${\rm Lip}(\tilde{v})\leq {\rm Lip}(v)+\mu=K_2+\mu$.
Thus, the function $F$ given by Theorem~\ref{thm-main-1} is a $C^{1,1}$-convex extension of $(\tilde{\alpha}(s), \tilde{v}(s))$
satisfying $F|_{S}=\tilde{\alpha}$, $\nabla F|_{S}=\tilde{v}$ and
$ \mathrm{Lip}(\nabla F)\leq (\mu +K_{2})^{2}(\mu-2K_{1})^{-1}.$
Therefore
$G(x)=F(x)-\frac{\mu}{2}|x|^2$
satisfies $G|_{S}=\alpha,$ $\nabla G|_{S}=v.$

Moreover, $G$ is
$\left(\frac{(\mu +K_{2})^{2}}{\mu-2K_{1}}-\mu\right)$-semiconcave and $\mu$-semiconvex (since $F$ is convex). We deduce 
$$
\mathrm{Lip}(\nabla G)\leq\max\left\{  \mu, \frac{(\mu +K_{2})^{2}}{\mu-2K_{1}}- \mu\right\}  .
$$
Minimizing the above quantity on $\mu\in (2 K_{1},+\infty )$ yields
\begin{eqnarray*}
\mathrm{Lip}(\nabla G)&\leq&
\min_{\mu\in (2 K_{1},\infty )}\max\left\{  \mu, \frac{(\mu +K_{2})^{2}}{\mu-2K_{1}}- \mu\right\}
\\
&=& \overline{\mu}:=   2K_1 + K_2 + \sqrt{(2K_1+K_2)^2+K_2^2}. 
\end{eqnarray*}
By Lemma~\ref{technique_cst} (Appendix), we have
$\max\, \{K_2, 4K_1-2K_2\} \leq  \Gamma^{1}(S,(\alpha,v))$. It follows that
$\mathrm{Lip}(\nabla G)\leq \left(\frac{5+\sqrt{29}}{2}\right)\Gamma^{1}(S,(\alpha,v))$. 
By~\cite[Proposition 2.4]{legruyer09}, we have $\mathrm{Lip}(\nabla G)= \Gamma^{1}(\mathcal{H},(G,\nabla G)).$
The result follows.
\end{proof}


\section{Limitations of the sup-inf approach}

The main result (Theorem~\ref{thm-main-2}) is heavily based on the explicit construction of a $C^{1,1}$-convex extension
of a 1-Taylor field $(\alpha,v)$ satisfying \eqref{form-azagra}, which in turn, is based on the sup-inf convolution approach.
The reader might wonder whether our approach can be adapted to include cases where less regularity is required, as for instance
$C^{1,\theta}$-extensions, that is, extensions to a $C^{1}$-function whose derivative has a
H\"older modulus of continuity with exponent $\theta\in(0,1)$.  The existence of such \emph{convex} extensions (and even $C^{1,\omega}$ convex extensions 
with a general modulus of continuity $\omega$) was established in finite dimensions in Azagra-Mudarra~\cite{am17bis} by means of involved arguments. Indeed, it would be natural to endeavor an adaptation of formula~\eqref{ext-C11-f} to treat the problem of $C^{1,\theta}$-convex extensions, for $0<\theta<1$.  According to~\cite{am17bis}, the adequate condition, analogous to~\eqref{form-azagra}, is that the 1-Taylor field has to satisfy, for some $M>0$,
 \begin{eqnarray}\label{form-azagra-holder}
\alpha(s_{2})\geq\alpha(s_{1})+\langle v(s_{1}),s_{2}-s_{1}\rangle
+\frac{\theta}{(1+\theta)M^{1/\theta}}|v(s_1)-v(s_2)|^{1+\frac{1}{\theta}}.
\end{eqnarray}
Unfortunately, the technique developed in Section~\ref{ss-2} is specific to the $C^{1,1}$-regularity and cannot be easily adapted to this more general case. Let us briefly explain the reason. \smallskip

\noindent Considering the suitable sup and inf-convolutions
\begin{eqnarray*}
f^\e(x)= \mathop{\rm sup}_{y\in \mathcal{H}}\left \{f(y)-\frac{|y-x|^{1+\theta}}{(1+\theta)\e^\theta}\right \},
\quad
f_\e(x)= \mathop{\rm inf}_{y\in \mathcal{H}}\left \{f(y)+\frac{|y-x|^{1+\theta}}{(1+\theta)\e^\theta}\right \},
\end{eqnarray*}
all of the arguments of the proof of Theorem~\ref{thm-main-1} go through except~\eqref{form-cruciale},
which fails to hold in this general case. More precisely, the convex extension $f$ defined by~\eqref{fcvx} satisfies
\begin{eqnarray}\label{f-line-holder}
l(x)\leq f(x)\leq q(x) \qquad \text{for all $x\in \mathcal{H}$ and $s\in S,$}
\end{eqnarray}
with equalities for $x=s$, where
\begin{eqnarray}
&& l(x):= \alpha(s)+ \langle v(s), x-s\rangle \label{linear-theta}\\
&& q(x):= \alpha(s)+ \langle v(s), x-s\rangle \label{parab-theta}
+\frac{M}{1+\theta}|x-s|^{1+\theta}.
\end{eqnarray}
Therefore for every $\e>0$ such that $M\e^{\theta} <1,$ we have
\begin{eqnarray*}
l(x)\leq (f^\e)_\e(x)\leq (q^\e)_\e(x)\,.
\end{eqnarray*}
Nonetheless, we may now possibly have
\begin{eqnarray}\label{contr-theta}
  (q^\e)_\e(s)\not= q(s),
\end{eqnarray}
yielding that $(f^\e)_\e$ is a  $C^{1,\theta}$-convex function but may differ from $f$ on $S$, hence it is not an extension of the latter. 
Let us underline that the problem arises even in dimension 1 and even for small $\e.$  In particular, the sup-convolution $q^\e$ may develop singularities
for arbitrary small $\e$ so that $q^\e$ is not anymore in the same class as $q$, contrary to the quadratic case (see~\eqref{sup-conv-quad}).

\begin{remark}   Recalling~\cite{ll86} that $u(x,t):=q^t(x)$ is a viscosity solution
  to the Hamilton-Jacobi equation
  $\partial_t u-\frac{\theta}{1+\theta}|\nabla u|^{1+\frac{1}{\theta}}=0$
  in $\mathcal{H}\times (0,\e),$ we obtain an explicit example where the solutions develop singularities
  instantaneously, even when starting with a $C^{1,\theta}$ initial condition $u(x,0)=q(x).$ See~\cite{bcjs99} for related comments.
\end{remark}

\begin{proof}[Sketch of proof of the Claim~\eqref{contr-theta}]
Without loss of generality we may assume that $\alpha(s)=0$ and $s=0$ in~\eqref{linear-theta}--\eqref{parab-theta}.
Fix $v\not= 0.$ Assume by contradiction that $(q^\e)_\e(0)=q(0)=l(0)=0.$ Then, since
$q$ is a $C^{1,\theta}$ function, necessarily, $\nabla (q^\e)_\e(0)=\nabla l(0)=v.$
Using that $y\mapsto q^\e(y)+\frac{|y|^{1+\theta}}{(1+\theta)\e^\theta}$
is a strictly convex function achieving a unique minimum $\overline{y}$ in $\mathcal{H},$
we obtain that
\begin{eqnarray*}
 && (q^\e)_\e(0) = q^\e (\overline{y}) +\frac{|\overline{y}|^{1+\theta}}{(1+\theta)\e^\theta}
  = \mathop{\rm sup}_{y\in \mathcal{H}} \left \{ q(y)-\frac{|y-\overline{y}|^{1+\theta}}{(1+\theta)\e^\theta}\right \}
  +\frac{|\overline{y}|^{1+\theta}}{(1+\theta)\e^\theta}\,=\,0,
\\
&&  \nabla (q^\e)_\e(0)= -\frac{\overline{y}|\overline{y}|^{\theta -1}}{\e^\theta}=v,
\end{eqnarray*}
yielding $\overline{y}= -\e v |v|^{\frac{1}{\theta}-1}\not= 0.$
To prove the claim, it is enough to find some $y\in\mathcal{H}$ such that
\begin{eqnarray*}
\varphi(y):=q(y)-\frac{|y-\overline{y}|^{1+\theta}}{(1+\theta)\e^\theta}
  +\frac{|\overline{y}|^{1+\theta}}{(1+\theta)\e^\theta} >0.
\end{eqnarray*}
In particular, let us seek for $y=\lambda \bar y$ where $\lambda\in\R$ is small. (Notice that this guarantees that
the computation would also hold when $\mathcal{H}$ is one dimensional.) We have
\begin{eqnarray*}
   (q^\e)_\e(0) &\geq& \varphi(y)= \frac{|\overline{y}|^{1+\theta}}{(1+\theta)\e^\theta}
  \left( M\e^\theta |\lambda|^{1+\theta} -(1+\theta)\lambda-|\lambda -1|^{1+\theta}+1\right)\\
  &=&  \frac{|\overline{y}|^{1+\theta}}{(1+\theta)\e^\theta}
  \left(  M\e^\theta |\lambda|^{1+\theta} -\frac{1}{2}(1+\theta)\theta\lambda^2 +o(\lambda^2) \right) >0=q(0),
\end{eqnarray*}
at least for small $\lambda >0$. 
\end{proof}

\section{Appendix}

\subsection{Independence of the Glaeser-Whitney constants in~\eqref{cond-GW}}

It is worth-noticing that there is no link between $K_1$ and $K_2$ in~\eqref{cond-GW}.
Each of these constants can be 0 while the other one can be very large, as it is shown in the following examples.

\begin{example}
Let~$\mathcal{H}=\R$, $S=\{0,1\}$ and $\alpha(0)=A>0$, $v(0)=0$, $\alpha(1)=0$, $v(1)=0$.
Then it follows that $K_1=A$ and $K_2=0$.  
\end{example}

\begin{example}
Let~$\mathcal{H}=\R^2$, $S=\{s_1,s_2\}$ with $s_1=(-1,0)$ and $s_2=(1,0)$. Set
$\alpha(s_1)=\alpha(s_2)=0$, $v(s_1)=(0,-A)$ and $v(s_2)=(0,A)$, for $A>0$. Since $\alpha(s_1)-\alpha(s_2)=0$ and $s_1-s_2\perp v(s_i)$, $i=1,2$,
we have $K_1=0$. Obviously $K_2=A$.  
\end{example}

\subsection{Inequality estimations between $\Gamma^{1}(S,(\alpha,v))$, $K_1$ and $K_2$.}

The following result has been used in the last part of the proof of Theorem~\ref{thm-main-2}

\begin{lemma}\label{technique_cst}
Let $\{(\alpha(s),v(s))\}_{s\in S}$ be a $1$-Taylor field satisfying the Glaeser-Whitney conditions~\eqref{cond-GW}. Then we have:\smallskip \\
(i) $K_2\leq  \Gamma^{1}(S,(\alpha,v))$; \smallskip \\
(ii) $4K_1-2K_2\leq  \Gamma^{1}(S,(\alpha,v))$.
\end{lemma}

\begin{proof}[Proof of Lemma~\ref{technique_cst}]
Recalling~\eqref{gamma1}, we have
\begin{eqnarray*}
  && \Gamma^{1}(S,(\alpha,v))\geq  \mathop{\rm sup}_{s_1\not= s_2} |B_{s_1s_2}|=K_2,
\end{eqnarray*}
which proves (i).\medskip

\noindent To establish (ii), we set
$$
K_1= \mathop{\rm sup}_{s_1\not= s_2} \frac{|k_1^{s_1s_2}|}{|s_1-s_2|^2}
\quad \text{with} \quad
k_1^{s_1s_2}=\alpha(s_2)-\alpha(s_1)-\langle v(s_1), s_2-s_1\rangle,
$$
and we deduce that
\begin{eqnarray*}
|-A_{s_1s_2}|
&=&
\left| \frac{2 k_1^{s_1s_2}}{|s_1-s_2|^2}+ \frac{\langle v(s_1)-v(s_2), s_2-s_1\rangle}{|s_1-s_2|^2}\right|
\\
&\geq &
\frac{2 |k_1^{s_1s_2}|}{|s_1-s_2|^2}-  \frac{|v(s_1)-v(s_2)|}{|s_1-s_2|}\\
&\geq &
\frac{2 |k_1^{s_1s_2}|}{|s_1-s_2|^2}- K_2.
\end{eqnarray*}
It follows
\begin{eqnarray*}
  && \Gamma^{1}(S,(\alpha,v))\geq  \mathop{\rm sup}_{s_1\not= s_2} 2|A_{s_1s_2}|
  \geq \mathop{\rm sup}_{s_1\not= s_2}\frac{4 |k_1^{s_1s_2}|}{|s_1-s_2|^2}- 2K_2 = 4K_1-2K_2.
\end{eqnarray*}
This completes the proof. \end{proof}

\vspace{0.3cm}

\noindent\textbf{Acknowledgement.} This work was partially supported by the Centre Henri Lebesgue ANR-11-LABX-0020-01.
Major part of this work has been done during a research visit of the first author to INSA Rennes. This author
is indebted to his hosts for hospitality. Research of Aris Daniilidis was
partially supported by BASAL PFB-03, FONDECYT grant 1171854 (Chile) and MTM2014-59179-C2-1-P
grant of MINECO (Spain) and ERDF (EU).


\vspace{1cm}

\noindent Aris Daniilidis
\smallskip

\noindent DIM--CMM, UMI CNRS 2807\newline Beauchef 851 (Torre Norte, piso~5),
Universidad de Chile, Santiago de Chile.\newline
\noindent E-mail: \texttt{arisd@dim.uchile.cl} \newline\noindent
\texttt{http://www.dim.uchile.cl/\symbol{126}arisd} \smallskip

\noindent Research supported by the grants: \newline BASAL PFB-03 (Chile),
FONDECYT 1171854 (Chile) and MTM2014-59179-C2-1-P (MINECO of Spain and ERDF of
EU).\newline

\bigskip

\noindent Mounir Haddou, Olivier Ley, Erwan Le Gruyer
\smallskip

\noindent IRMAR, INSA Rennes, CNRS UMR 6625 \newline
20 avenue des Buttes de Coesmes, F-35708 Rennes, France\newline
\noindent E-mail: \texttt{\{mounir.haddou, olivier.ley, erwan.le-gruyer\}\,\,@insa-rennes.fr} \newline\noindent
\texttt{http://\{haddou, ley\}.perso.math.cnrs.fr/}

\end{document}